\documentclass[12pt,a4paper]{amsart}
\usepackage[active]{srcltx}
\usepackage[all]{xy}
\usepackage{amsmath,amsthm,amssymb,latexsym,epic,bbm,comment}
\usepackage{graphicx,enumerate,stmaryrd,xspace}
\usepackage{color}

\newtheorem{theorem}{Theorem}
\newtheorem{example}[theorem]{Example}
\newtheorem{remark}[theorem]{Remark}
\newtheorem{lemma}[theorem]{Lemma}
\newtheorem{proposition}[theorem]{Proposition}
\newtheorem{corollary}[theorem]{Corollary}
\numberwithin{equation}{section}

\begin{document}

\title[Weight modules over  Weyl algebras]{Weight modules over infinite dimensional Weyl algebras}
\author[V.~Futorny, D.~Grantcharov and V.~Mazorchuk]{Vyacheslav Futorny, Dimitar Grantcharov\\ and Volodymyr Mazorchuk}

\begin{abstract}
We classify simple weight modules over infinite dimensional Weyl algebras and realize them using the action on 
certain localizations of the polynomial ring. We describe indecomposable projective and injective weight modules 
and deduce from this a description of blocks of the category of weight modules by quivers and relations. As a 
corollary we establish Koszulity for all blocks. \\

\noindent 2000 MSC: 17B10, 17B65, 16D60
\end{abstract}
\maketitle

\section{Introduction}\label{s0}

Weyl algebras are classical objects of study in representation theory that arise naturally both in mathematics 
and physics and have many important applications. For example, an essential part of the ${\mathcal D}$-module
approach to the representation theory of a finite dimensional simple Lie algebra $\mathfrak{g}$ is existence 
of a natural homomorphism from $U(\mathfrak{g})$ to a finite dimensional (or finite rank) Weyl algebra $A_n$. In the case of an affine Lie 
algebra there is a similar homomorphism, but now to an infinite dimensional Weyl algebra ${ A}_{\infty}$. 
A realization of the standard Verma $\widehat{\mathfrak{sl}}_2$-modules as ${ A}_{\infty}$-modules was
given by Wakimoto in \cite{W} and was later generalized to other types of Verma modules and higher rank 
affine Lie algebras, see \cite{C,CF,FF1,FF2,JK,V}.

The algebra ${ A}_{\infty}$ can be viewed as an infinite rank generalized Weyl algebra
(in the sense of \cite{Ba,BB}) and thus has a natural category of representations consisting of the
so-called {\em weight} modules. Such modules over finite and infinite dimensional Weyl algebras have 
been extensively studied in the last 20 years. Various constructions and classification results appear  
in \cite{BB,BBF,GS}. In particular, a partial classification of simple weight 
${ A}_{\infty}$-modules was given in \cite{BBF}. On the other hand, new examples of 
simple weight ${ A}_{\infty}$-modules recently appeared in \cite{MZ}.

In the present paper we complete classification of simple weight modules over infinite dimensional Weyl algebras.
The Weyl algebra is usually defined as the algebra of certain differential operators of a polynomial algebra.
This action localizes to Laurent polynomials and then can be twisted to all polynomials with not
necessarily integer exponents. We use the latter action both for our main classification result,
Theorem~\ref{thm3}, and also to give an explicit realization of all simple weight modules in Subsection~\ref{s1.5}.
This explicit realization is helpful for numerous reasons. Firstly, it can be used to construct new  free field
representations of affine Lie algebras and direct limit Lie algebras. Also, one would now expect realizations
of these representations  in terms of sections of vector bundles on flag varieties. Furthermore, this
realization allows us to describe indecomposable projective and injective weight modules, see
Subsections~\ref{s1.4},\ref{s1.45} and \ref{s1.9} which, in turn, leads to an explicit description,
via quivers and relations, for all blocks of the category of weight modules. All details of the latter 
are collected in Section~\ref{s2} with the main result being Theorem~\ref{thm27}. As a  
consequence of this description we show that all blocks correspond to Koszul algebras, see Proposition~\ref{prop21}. 
\vspace{1mm}

\noindent
{\bf Acknowledgements.} D.G. gratefully acknowledges the
hospitality and excellent working conditions at the S\~ao Paulo
University where most of this work was completed. V.F. was
supported by the CNPq grant (301743/2007-0) and by the
Fapesp grant (2010/50347-9). D.G was supported by the
Fapesp grant (2011/21621-8) and by the NSA grant
H98230-10-1-0207. V.M. was supported by the Royal Swedish Academy of Sciences and by the Swedish Research Council.

\section{Generalities}\label{s1}

\subsection{The infinite dimensional Weyl algebra}\label{s1.1}

Let $\Bbbk$ be an algebraically closed field of 
characteristic zero and $\mathtt{I}$ be an infinite set
satisfying $|I|<|\Bbbk|$ (this assumption will be essential,
in particular, for Lemma~\ref{lem1}). Consider the commutative
$\Bbbk$-algebra $B$ of polynomials in infinitely many
variables $x_i$, $i\in \mathtt{I}$. For $i\in \mathtt{I}$
let $X_i:=x_i\cdot$ be the linear operator on $B$ given by 
multiplication with $x_i$. Let further $Y_i:=\partial_i$
denote the linear operator on $B$ given by 
the partial derivative with respect to $x_i$, that is
$\partial_i$ is a derivation of $B$ defined by
$\partial_i\cdot x_j=\delta_{i,j}$,
the Kronecker delta, on the generators.

Consider the additive abelian group 
$\Bbbk^{\mathtt{I}}$ consisting of all
vectors $\mathbf{v}=(v_i)_{i\in \mathtt{I}}$ 
with coefficients from $\Bbbk$.
Let $\mathbb{Z}_f^{\mathtt{I}}$ denote the 
subgroup of all integral vectors with at most finitely many
nonzero coefficients.
For $\mathbf{v}\in \mathbb{Z}_f^{\mathtt{I}}$ we denote by 
$\mathbf{x}^{\mathbf{v}}$ the monomial 
$\prod_{i\in \mathtt{I}}x_i^{v_i}\in B$.
The monomials $\{\mathbf{x}^{\mathbf{v}}:\mathbf{v}\in \mathbb{Z}_f^{\mathtt{I}}\}$
form the {\em standard basis} in $B$ (over $\Bbbk$).

The {\em infinite Weyl algebra} $A=A_{\Bbbk,I}$ is the
subalgebra in $\mathrm{End}_{\Bbbk}(B)$ generated by 
all $X_i$ and $Y_i$, $i\in \mathtt{I}$. It is easy to check
that these generators satisfy the following relations:
\begin{displaymath}
\begin{array}{cccc}
X_iX_j=X_jX_i,& \forall i,j;& 
Y_iY_j=Y_jY_i,& \forall i,j;\\ 
X_iY_j=Y_jX_i,& i\neq j;& 
X_iY_i=Y_iX_i-\mathrm{Id},& \forall i, 
\end{array}
\end{displaymath}
where $\mathrm{Id}$ denotes the identity linear 
transformation.
It is easy to check that these relations give
a presentation for $A$.

\subsection{A maximal commutative subalgebra}\label{s1.2}

For $i\in \mathtt{I}$ set $t_i:=X_iY_i\in A$ and denote
by $A_0$ the subalgebra of $A$ generated by all  $t_i$'s.
From the presentation of $A$ it follows that $A_0$ is 
a commutative subalgebra of $A$ and that the generators
$t_i$ are algebraically independent. 

For $\mathbf{v}\in \mathbb{Z}_f^{\mathtt{I}}$ consider the
element 
\begin{displaymath}
\mathbf{X}_{\mathbf{v}}:=
\prod_{i:v_i>0}X_i^{v_i}\prod_{i:v_i<0}Y_i^{-v_i}. 
\end{displaymath}
From the presentation of $A$ it follows that $A$
is free both as a left and as a right $A_0$-module
with basis $\{\mathbf{X}_{\mathbf{v}}:\mathbf{v}\in \mathbb{Z}_f^{\mathtt{I}}\}$.
Since none of $\mathbf{X}_{\mathbf{v}}$, $\mathbf{v}\neq 0$,
commutes with all elements of $A_0$, it follows that
$A_0$ is a maximal commutative subalgebra in $A$.

For $\mathbf{p}\in \Bbbk^{\mathtt{I}}$ denote by 
$\mathfrak{m}_{\mathbf{p}}$ the maximal ideal in
$A_0$ generated by $t_i-p_i$, $i\in \mathtt{I}$.

\begin{lemma}\label{lem1}
Every maximal ideal in $A_0$ has the form 
$\mathfrak{m}_{\mathbf{p}}$ for some 
$\mathbf{p}\in \Bbbk^{\mathtt{I}}$.
\end{lemma}

\begin{proof}
Let $\mathfrak{m}$ be a maximal ideal in $A_0$. Then
$A_0/\mathfrak{m}$ is a field extending $\Bbbk$. As
$\Bbbk$ is algebraically closed, we either have that the
extension $\Bbbk\hookrightarrow A_0/\mathfrak{m}$ is
an isomorphism, in which case $\mathfrak{m}=\mathfrak{m}_{\mathbf{p}}$ 
for some  $\mathbf{p}\in \Bbbk^{\mathtt{I}}$, or this extension
is purely transcendental. In the latter case take any 
$z\in (A_0/\mathfrak{m})\setminus \Bbbk$. Then the
elements $(z-c)^{-1}$, $c\in\Bbbk$, are linearly
independent over $\Bbbk$ which means that the 
$\Bbbk$-dimension of $A_0/\mathfrak{m}$ is at least 
$|\Bbbk|$. On the other hand, the $\Bbbk$-dimension 
of $A_0$ is $|I|<|\Bbbk|$, a contradiction. The claim follows.
\end{proof}

\section{Weight modules}\label{s2}

\subsection{Weight $A$-modules}\label{s1.3}

Let $M$ be an $A$-module and $\mathfrak{m}$ be a
maximal ideal in $A_0$. An element $x\in M$ is called
{\em weight}  if 
$\mathfrak{m}\cdot x=0$. The module $M$ is called 
a {\em weight} module if it has a basis consisting of
weight elements. By Lemma~\ref{lem1}, the letter is
equivalent to 
$\displaystyle M=\bigoplus_{\mathbf{p}\in \Bbbk^{\mathtt{I}}}M_{\mathbf{p}}$,
where
\begin{displaymath}
M_{\mathbf{p}}:=\{x\in M:t_i\cdot x=p_i x, \forall
i\in I\} 
\end{displaymath}
is the {\em weight space} of {\em weight} $\mathbf{p}$.
For a weight module $M$ the set of all 
$\mathbf{p}\in \Bbbk^{\mathtt{I}}$ such that 
$M_{\mathbf{p}}\neq 0$ is called the {\em support}
of $M$ and denoted $\mathrm{supp}(M)$.

Denote by $\mathfrak{W}$ the full subcategory in
$A\text{-}\mathrm{mod}$, the category of all $A$-mo\-dules, consisting of all weight $A$-modules. For
$\mathbf{p}\in \Bbbk^{\mathtt{I}}$ we denote by
$\mathfrak{W}_{\mathbf{p}}$ the full subcategory of 
$\mathfrak{W}$ consisting of all $M$ satisfying
$\mathrm{supp}(M)\subset \mathbf{p}+ \mathbb{Z}_f^{\mathtt{I}}$. Clearly,
$\mathfrak{W}_{\mathbf{p}}=\mathfrak{W}_{\mathbf{m}}$
if and only if $\mathbf{m}\in \mathbf{p}+ \mathbb{Z}_f^{\mathtt{I}}$.
Further, we have the usual decomposition
\begin{displaymath}
\mathfrak{W}\cong \bigoplus_{\xi\in \Bbbk^{\mathtt{I}}/\mathbb{Z}_f^{\mathtt{I}}}
\mathfrak{W}_{\mathbf{p}_{\xi}},
\end{displaymath}
where $\mathbf{p}_{\xi}\in \xi$ is some fixed representative.

\subsection{Projective weight $A$-modules}\label{s1.4}

For $\mathbf{p}\in \Bbbk^{\mathtt{I}}$
denote by $\mathfrak{V}_{\mathbf{p}}$ the full subcategory 
in $A_0\text{-}\mathrm{mod}$ consisting of all 
$A_0$-modules $N$ such that $\mathfrak{m}_{\mathbf{p}}\cdot N=0$. The category  $\mathfrak{V}_{\mathbf{p}}$ is semisimple
and has a unique (up to isomorphism) simple object
$\Bbbk_{\mathbf{p}}:=A_0/\mathfrak{m}_{\mathbf{p}}$
which is also projective.

The restriction functor $\mathrm{Res}_{\mathbf{p}}:
\mathfrak{W}\to \mathfrak{V}_{\mathbf{p}}$ is obviously
exact and is right adjoint to the induction functor
\begin{displaymath}
\mathrm{Ind}_{\mathbf{p}}:= A\otimes_{A_0}{}_-: 
\mathfrak{V}_{\mathbf{p}}\to \mathfrak{W}. 
\end{displaymath}
The functor $\mathrm{Ind}_{\mathbf{p}}$  is exact as $A$ is free over $A_0$,
moreover, being left adjoint to an exact functor,
the functor $\mathrm{Ind}_{\mathbf{p}}$ maps
projective objects to projective objects. Therefore
$P(\mathbf{p}):=\mathrm{Ind}_{\mathbf{p}}\Bbbk_{\mathbf{p}}\cong A/A\mathfrak{m}_{\mathbf{p}}$
is projective in $\mathfrak{W}$. Moreover, by 
adjunction, for any $M\in \mathfrak{W}$ we have
a natural isomorphism
\begin{equation}\label{eq37}
\mathrm{Hom}_{\mathfrak{W}}(P(\mathbf{p}),M)
\cong M_{\mathbf{p}}.
\end{equation}

Denote by $\sigma_i$
the automorphism of $A_0$ given by 
\begin{displaymath}
\sigma_i(t_j)=\begin{cases}
t_j,&i\neq j;\\ t_i+1,& i=j.
\end{cases}
\end{displaymath}
Let $H$ be the group generated by $\sigma_i$, $i\in\mathtt{I}$. Identifying maximal ideals of
$A_0$ with elements in $\Bbbk^{\mathtt{I}}$ yields
a canonical identification $H\cong \mathbb{Z}_f^{\mathtt{I}}$. If $M$ is a weight
$A$-module, then the defining relations of $A$ imply
that for any $\mathbf{p}\in \Bbbk^{\mathtt{I}}$ 
we have $X_i M_{\mathbf{p}}\subset M_{\sigma_i(\mathbf{p})}$
and $Y_i M_{\mathbf{p}}\subset M_{\sigma_i^{-1}(\mathbf{p})}$.

The group $\mathbb{Z}_f^{\mathtt{I}}$ acts freely
on $\Bbbk^{\mathtt{I}}$. Moreover, elements of $\mathbb{Z}_f^{\mathtt{I}}$
naturally index the free basis of $A$ as a right $A_0$-module.
It follows that 
\begin{displaymath}
P(\mathbf{p})\cong\bigoplus_{\sigma\in \mathbb{Z}_f^{\mathtt{I}}}
P(\mathbf{p})_{\sigma(\mathbf{p})},
\end{displaymath}
that all $\sigma(\mathbf{p})$ are different and all
$P(\mathbf{p})_{\sigma(\mathbf{p})}$ are one-dimensional
over $\Bbbk$. 

\begin{proposition}\label{prop2}
For any $\mathbf{p}\in \Bbbk^{\mathtt{I}}$ we have:
\begin{enumerate}[$($a$)$]
\item\label{prop2.1} $P(\mathbf{p})$ is indecomposable
and has simple top $L(\mathbf{p})$,
\item\label{prop2.2} $\dim_{\Bbbk}L(\mathbf{p})_{\mathbf{p}}=1$ and 
$L(\mathbf{p})$ is the unique (up to isomorphism) 
simple module with this property.
\end{enumerate}
\end{proposition}

\begin{proof}
Using \eqref{eq37}, we have $\mathrm{End}_{\mathfrak{W}}(P(\mathbf{p}))\cong
P(\mathbf{p})_{\mathbf{p}}\cong \Bbbk$. This implies
claim \eqref{prop2.1}. As $P(\mathbf{p})$ is generated by
$P(\mathbf{p})_{\mathbf{p}}$, any proper submodule of
$P(\mathbf{p})$ does not intersect $P(\mathbf{p})_{\mathbf{p}}$. On the other hand,
let $N$ denote the sum of all
submodules of $P(\mathbf{p})$ which do not intersect
$P(\mathbf{p})_{\mathbf{p}}$. Then $N$ is a proper
submodule of $P(\mathbf{p})$ and hence is the unique
maximal submodule. Then $P(\mathbf{p})/N$ is the unique
simple quotient of $P(\mathbf{p})$
and  $(P(\mathbf{p})/N)_{\mathbf{p}}\neq 0$
implying claim 
\eqref{prop2.2}.
\end{proof}
 Note that $L(\mathbf{0})$ is isomorphic to the defining representation $B$ of $A$, where all component of $\mathbf{0}$ are zero.

\subsection{Involution}\label{s1.45}

Denote by $\mathfrak{W}^f$ the full subcategory of
$\mathfrak{W}$ consisting of all modules with finite
dimensional weight spaces. From the previous subsection
we have that both $P(\mathbf{p})$ and $L(\mathbf{p})$
belong to $\mathfrak{W}^f$ for all 
$\mathbf{p}\in \Bbbk^{\mathtt{I}}$. 

The algebra $A$ has the standard involutive anti-automorphism 
$\diamond$ defined on the generators via $X_i^{\diamond}=Y_i$.
Note that $\diamond$ fixes $A_0$ pointwise. The category 
$\mathfrak{W}^f$ has the standard restricted duality
$\star$ defined as follows: for $M\in \mathfrak{W}^f$
we have $\displaystyle M^{\star}:=
\bigoplus_{\mathbf{p}\in \Bbbk^{\mathtt{I}}}
\mathrm{Hom}_{\Bbbk}(M_{\mathbf{p}},\Bbbk)$
with the left action of $A$ defined via $\diamond$,
and with the obvious action on morphisms. As
$\diamond$ fixes $A_0$ pointwise, we have
$\mathrm{supp}(M^{\star})=\mathrm{supp}(M)$ which
implies that $L(\mathbf{p})^{\star}\cong 
L(\mathbf{p})$. It follows that 
$P(\mathbf{p})^{\star}$ is the injective envelope of
$L(\mathbf{p})$.

\subsection{Realization via polynomial action}\label{s1.5}

The defining action of $A$ on $B$ admits the 
following obvious generalization. For
$\mathbf{p}\in \Bbbk^{\mathtt{I}}$ set
$\mathbf{x}^{\mathbf{p}}:=\prod_{i\in\mathtt{I}}x_i^{p_i}$
and let $B(\mathbf{p})$ be the linear span of 
$\mathbf{x}^{\mathbf{m}}$, $\mathbf{m}\in \mathbf{p}+\mathbb{Z}_f^{\mathtt{I}}$.
Define the action of $X_i$ on $B(\mathbf{p})$ by multiplication with $x_i$ and the action of $Y_i$ on $B(\mathbf{p})$ by partial derivative with respect to $x_i$.
It is straightforward to check that this defines on
$B(\mathbf{p})$ the structure of a weight $A$-module,
that $\mathrm{supp}(B(\mathbf{p}))=\mathbf{p}+\mathbb{Z}_f^{\mathtt{I}}$
and that all nonzero weight spaces of $B(\mathbf{p})$
are one-dimensional. 

\begin{proposition}\label{prop31}
\begin{enumerate}[$($a$)$]
\item\label{prop31.1} Let $\mathbf{p}\in \Bbbk^{\mathtt{I}}$ be such that $p_i\in\mathbb{Z}$ implies 
$p_i\in\{0,1,2,\dots\}$ for all $i$. Then $B(\mathbf{p})$ and $B(\mathbf{p})^{\star}$ are the injective envelope 
and the projective cover of $L(\mathbf{p})$, respectively.
\item\label{prop31.2} Let $\mathbf{p}\in \Bbbk^{\mathtt{I}}$ be such that $p_i\in\mathbb{Z}$ implies 
$p_i\in\{-1,-2,\dots\}$ for all $i$. Then $B(\mathbf{p})$ and $B(\mathbf{p})^{\star}$ are the projective cover 
and the injective envelope of $L(\mathbf{p})$, respectively.
\end{enumerate}
\end{proposition}

\begin{proof}
We prove claim \eqref{prop31.2}. Claim \eqref{prop31.1} is proved similarly.  Let $\mathbf{p}\in \Bbbk^{\mathtt{I}}$ 
be such that $p_i\in\mathbb{Z}$ implies $p_i\in\{-1,-2,\dots\}$ for all $i$. As $\star$ is a duality, it is enough
to show that $B(\mathbf{p})\cong P(\mathbf{p})$. By construction, $B(\mathbf{p})_{\mathbf{p}}\neq 0$ which gives
a nonzero homomorphism $P(\mathbf{p})\to B(\mathbf{p})$ by \eqref{eq37}. This map is easily seen to be 
injective. Comparing the characters of $P(\mathbf{p})$ and $B(\mathbf{p})$, it follows that this map is bijective.
\end{proof}

Consider the set 
\begin{displaymath}
S:=\{\mathbf{m}\in \mathbf{p}+\mathbb{Z}_f^{\mathtt{I}}:
\forall i\,\, p_i\in\{0,1,2,\dots\}\Rightarrow
m_i\in\{0,1,2,\dots\}\} 
\end{displaymath}
and let $\displaystyle N:=
\bigoplus_{\mathbf{m}\in S} {B}(\mathbf{p})_{\mathbf{m}}$.
Then $N$ is obviously a submodule of ${B}(\mathbf{p})$.
Consider the set 
\begin{displaymath}
S':=\{\mathbf{m}\in S:
\forall i\,\, p_i\in\{-1,-2,\dots\}\Rightarrow
m_i\in\{0,1,2,\dots\}\} 
\end{displaymath}
and let $\displaystyle N':=
\bigoplus_{\mathbf{m}\in S'} {B}(\mathbf{p})_{\mathbf{m}}$.
Then $N'$ is obviously a submodule of $N$.

\subsection{Simple weight modules}\label{s1.6}

For $\mathbf{p}\in \Bbbk^{\mathtt{I}}$ denote by 
$\overline{\mathbf{p}}$ the set of all 
$\mathbf{k}\in \mathbf{p}+\mathbb{Z}_f^{\mathtt{I}}$ which satisfy the
following conditions for all $i\in\mathtt{I}$:
\begin{displaymath}
\begin{array}{lcl}
p_i\in\{-1,-2,-3,\dots\}&\text{ implies }&
k_i\in\{-1,-2,-3,\dots\};\\
p_i\in\{0,1,2,3,\dots\}&\text{ implies }&
k_i\in\{0,1,2,3,\dots\}.
\end{array}
\end{displaymath}
In the notation from the previous subsection we have
$\mathrm{supp}(N/N')=\overline{\mathbf{p}}$.

\begin{proposition}\label{prop3}
For $\mathbf{p}\in \Bbbk^{\mathtt{I}}$ we have:
\begin{enumerate}[$($a$)$]
\item\label{prop3.1} $\mathrm{supp}(L(\mathbf{p}))=
\overline{\mathbf{p}}$;
\item\label{prop3.2} if $\mathbf{s}\in \Bbbk^{\mathtt{I}}$,
then $L(\mathbf{p})\cong L(\mathbf{s})$ if and only if
$\mathbf{s}\in \overline{\mathbf{p}}$.
\end{enumerate}
\end{proposition}

\begin{proof}
First we show that the module $N/N'$ constructed in the 
previous subsection is simple. Let $\mathbf{m}\in
\overline{\mathbf{p}}$ and $v\in (N/N')_{\mathbf{m}}$
be a nonzero element. We have to show that $v$ generates
$N/N'$. For this it is enough to check that, given
$i\in \mathtt{I}$, we have $X_i v\neq 0$ if
$m_i\neq -1$ and we have $Y_i v\neq 0$ if
$m_i\neq 0$. However, both these claims follow directly from
the definitions.

As $(N/N')_{\mathbf{p}}\neq 0$ by construction,
we have $N/N'\cong L(\mathbf{p})$ by 
Proposition~\ref{prop2}\eqref{prop2.2}. Now claim
\eqref{prop3.1} follows by construction and claim 
\eqref{prop3.2} follows from claim  \eqref{prop3.1}
and Proposition~\ref{prop2}.
\end{proof}

The above results can now be summarized as follows:
let $\sim$ be the equivalence relation on $\Bbbk^{\mathtt{I}}$
defined as follows: $\mathbf{p}\sim\mathbf{m}$ if and
only if $\overline{\mathbf{p}}=\overline{\mathbf{m}}$.

\begin{theorem}[Classification of simple weight 
$A$-modules]\label{thm3}
The map $\overline{\mathbf{p}}\mapsto L(\mathbf{p})$,
$\mathbf{p}\in \Bbbk^{\mathtt{I}}$, is
a bijection between $\Bbbk^{\mathtt{I}}/\sim$ and the
set of isomorphism classes of simple weight 
$A$-modules.
\end{theorem}

\begin{example}\label{ex9}
{\rm 
Take $\mathbf{p}=(1,1,....)$, i.e. $p_i = 1$ for every $i$. The simple module
$L(\mathbf{p})$ locally finite with respect to each $\partial_i$, however, there is no $v \in L(\mathbf{p})$ 
for which $\partial_i v = 0$ for every $i$. Similar examples were considered in \cite[Subsection 4.2]{MZ}.
}
\end{example}

\begin{remark}\label{rem8}
{\rm 
Theorem~\ref{thm3} transfers
mutatis mutandis to tensor products of rank one 
generalized Weyl algebras  in the sense of  \cite{Ba} associated to $\Bbbk[x]$ and
$\sigma:\Bbbk[x]\to \Bbbk[x]$ defined by $\sigma(x)=x+1$.
}
\end{remark}

\subsection{Twisted polynomial realizations}\label{s1.7}

For $\mathtt{J}\subset\mathtt{I}$ let $\theta_J$ denote the automorphism of  $A$ given by
\begin{eqnarray*}
\theta_{\mathtt{J}} (X_j)  =  Y_j, \theta_{\mathtt{J}}(Y_j) =  -X_j, &\mbox{ if }& j \in {\mathtt{J}};\\
\theta_{\mathtt{J}} (X_i)  =  X_i, \theta_{\mathtt{J}}(Y_i) = Y_i, & \mbox{ if }& i \notin {\mathtt{J}}.
\end{eqnarray*}
For $i\in \mathtt{I}$ we have
\begin{equation}\label{eq47}
\theta_{\mathtt{J}} (t_i)=\begin{cases}-t_i-1,&i\in \mathtt{J}\\t_i,& i\not\in \mathtt{J}. \end{cases}
\end{equation}
For an $A$-module $M$, the module obtained by  twisting the $A$-action on $M$  by $\theta_{\mathtt{J}}$ 
will be denoted by $M^{\theta_{\mathtt{J}}}$. From the above we have that $M^{\theta_{\mathtt{J}}}$ is a weight
module if and only if $M$ is. Furthermore, $\mathbf{p}\in \mathrm{supp}(M)$ if and only if 
$\theta_J (\mathbf{p}) \in \mathrm{supp}(M^{\theta_{\mathtt{J}}})$, where $\theta_J (\mathbf{p})_i=p_i$ if $i\not\in \mathtt{J}$ and
$\theta_J (\mathbf{p})_i=-p_i-1$ if $i\in \mathtt{J}$.

Denote by $\Bbbk^{\mathtt{I}}_+$ the set of all $\mathbf{p}\in \Bbbk^{\mathtt{I}}$ such that 
$p_i\in\mathbb{Z}$ implies $p_i\in\{0,1,2,\dots\}$ for all $i$. For $\mathbf{p}\in \Bbbk^{\mathtt{I}}_+$ denote by 
$\mathtt{J}_{\mathbf{p}}$ the set of all $i\in\mathtt{I}$ such that $p_i\in\mathbb{Z}$. From
Theorem~\ref{thm3} we immediately obtain the following:

\begin{corollary}\label{cor41}
\begin{enumerate}[$($a$)$]
\item\label{cor41.1} Let $L$ be a simple weight $A$-module. Then there are unique $\mathbf{p}\in \Bbbk^{\mathtt{I}}_+$ and
$\mathtt{J}\subset \mathtt{J}_{\mathbf{p}}$ such that $L\cong L(\mathbf{p})^{\theta_{\mathtt{J}}}$. In fact, if $L =  L(\mathbf{p})$, then $\mathtt{J} = \{ i \in \mathtt{I} : q_i \in \{-1,-2,... \} \}$ and 
$\mathbf{p} = \theta_{\mathtt{J}} (\mathbf{q})$.
\item\label{cor41.2} For $\mathbf{p},\mathbf{q}\in \Bbbk^{\mathtt{I}}_+$, $\mathtt{J}\subset \mathtt{J}_{\mathbf{p}}$
and $\mathtt{J}'\subset \mathtt{J}_{\mathbf{q}}$ we have $L(\mathbf{p})^{\theta_{\mathtt{J}}}\cong
L(\mathbf{q})^{\theta_{\mathtt{J}'}}$ if and only if $\mathbf{p}=\mathbf{q}$ and $\mathtt{J}=\mathtt{J}'$.
\end{enumerate}
\end{corollary}

Combining this with Proposition~\ref{prop31} we obtain:

\begin{corollary}\label{cor42}
Let $\mathbf{p}\in \Bbbk^{\mathtt{I}}_+$ and $\mathtt{J}\subset \mathtt{J}_{\mathbf{p}}$.
Then $B(\mathbf{p})^{\theta_{\mathtt{J}}}$ and $(B(\mathbf{p})^{\theta_{\mathtt{J}}})^{\star}$ 
are the injective envelope and the projective cover of $L(\mathbf{p})^{\theta_{\mathtt{J}}}$, respectively.
\end{corollary}

\subsection{Localization realizations}\label{s1.9}

Let $\mathtt{J}\subset\mathtt{I}$. The adjoint action of $X_i$ on $A$ is locally nilpotent and hence 
$X_i$, $i\in \mathtt{J}$, generate a multiplicative Ore subset of $A$ giving rise to the corresponding  
Ore localization ${\mathbf D}_J A$ of $A$. Define the functor $\mathrm{F}_J := 
\mathrm{Res}^{{\mathbf  D}_J A}_{A}\circ \mathrm{Ind}^{{\mathbf  D}_J A}_{A}$  on the category $A$-mod.

Furthermore, similarly to \cite[Lemma~4.3]{Ma}, the algebra ${\mathbf  D}_J A$ has a 
family of automorphisms $\varphi_{\mathbf{x}}$, $\mathbf{x}\in\Bbbk^{\mathtt{J}}$, which are polynomial in
components of $\mathbf{x}$ and such that for $\mathbf{x}\in\mathbb{Z}^{\mathtt{J}}$ and
for any $a\in A$ we have 
\begin{displaymath}
\varphi_{\mathbf{x}}(a)= \prod_{i\in \mathtt{J}}X_i^{-x_i}\,a\,\prod_{i\in \mathtt{J}}X_i^{x_i}
\end{displaymath}
(note that if $a$ is fixed, then it commutes with all but finitely many of the $X_i$'s and hence the expression 
makes sense by canceling out all other terms). For a ${\mathbf  D}_J A$-module $M$, denote by 
$M^{\varphi_{\mathbf{x}}}$ the ${\mathbf  D}_J A$-mo\-dule obtained from $M$ after twisting by $\varphi_{\mathbf{x}}$. Note that $\mathrm{supp}(M^{\varphi_{\mathbf{x}}}) = \mathrm{supp}(M) + \mathbf{x}$, where $\mathbf{x}$ is considered 
as an element of $\Bbbk^{\mathtt{I}}$ by setting all components in $\mathtt{I} \setminus \mathtt{J}$ to be zero.

\begin{proposition}\label{propnew}
Let  $\mathbf{p}\in \Bbbk^{\mathtt{I}}_+$.
\begin{enumerate}[$($a$)$]
\item\label{propnew.1}
$L(\mathbf{p}) \simeq \mathrm{F}^{\mathbf{p'}}_{\mathtt{I} \setminus \mathtt{J}_{\mathbf{p}}} L(\mathbf{0})$, where $\mathbf{p'} \in \Bbbk^{\mathtt{I} \setminus \mathtt{J}_{\mathbf{p}}}$ with $p_i'=p_i$.
\item\label{propnew.2} $B(\mathbf{p}) \simeq  \mathrm{F}_{\mathtt{J}_{\mathbf{p}}} L(\mathbf{p})$.
\end{enumerate}
\end{proposition}

\begin{proof}
By Theorem \ref{thm3},  $L \mapsto \mbox{supp} (L)$ defines a bijection between the set of simple weight 
$A$-modules and $\Bbbk^{\mathtt{I}}/\sim$. Since every simple module has weight multiplicities at most $1$, 
for a weight $A$-module $M$ we have that $M \simeq L(\mathbf{p})$ if and only if $\mbox{supp} (M) = 
\mbox{supp} (L(\mathbf{p}))$. On the other hand, $\mbox{supp}(\mathrm{F}^{\mathbf{p}}_{\mathtt{I} \setminus \mathtt{J}_{\mathbf{p}}} L(\mathbf{0})) = \overline{\mathbf{p}} =   \mbox{supp} (L(\mathbf{p}))$, 
which implies  claim \eqref{propnew.1}. For claim \eqref{propnew.2} we use the polynomial  
description of  $B(\mathbf{p})$ from Subsection \ref{s1.5}. A basis vector of 
$\mathrm{F}_{\mathtt{J}_{\mathbf{p}}} L(\mathbf{p})$ can be written uniquely in the form 
$\mathbf{x}^{\mathbf{m}} \mathbf{x}^{\mathbf{q}}$, where  $m_i \in \{0,-1,-2,... \}$, 
$i \in \mathtt{J}_{\mathbf{p}}$, $m_j = 0$, $j \notin \mathtt{J}_{\mathbf{p}}$,  and 
$ \mathbf{q} \in \overline{\mathbf{p}}$. Then  $\mathbf{x}^{\mathbf{m}} \mathbf{x}^{\mathbf{q}} \mapsto \mathbf{x}^{\mathbf{m+q}}$ 
defines an isomorphism $\mathrm{F}_{\mathtt{J}_{\mathbf{p}}} L(\mathbf{p}) \simeq B(\mathbf{p})$. 
\end{proof}

 Note that due to Corollary \ref{cor41}, the injective envelope of every simple weight module $L$ is isomorphic to the localization of $L$ relative to an appropriate Ore's multiplicative subset of $A$. Presenting the indecomposable injectives as localizations of their simple submodules is an idea explored for  categories of weight modules of the symplectic Lie algebras in \cite{GS}.

\section{Quiver of $\mathfrak{W}$}\label{s3}

\subsection{Explicit description of the quiver}\label{s3.1}

For a nonempty set $E$ define the quiver $Q=Q_E$ as follows:
The vertices of $Q$ are all finite subsets of $E$. 
For $U,W\in Q$ there is one arrow from $U$ to $W$ if and 
only if the symmetric difference $U\triangle W$ is a singleton. 
Impose in $Q$ the following relations: if $U$ and $W$ are such 
that there are arrows $\alpha:U\to W$ and $\beta:W\to U$,
then $\alpha\beta=\beta\alpha=0$; if $U,W,U',W'$ are
different subsets such that there are arrows
$\alpha:U\to W$, $\beta:W\to W'$, $\alpha':U\to U'$
and $\beta':U'\to W'$, then $\beta\alpha=\beta'\alpha'$.
For example, if $E=\{1,2\}$, then the corresponding quiver $Q$
looks as follows:
\begin{displaymath}
\xymatrix{
\varnothing\ar@/^/[rr]^{\eta}\ar@/_/[d]_{\xi}&& \{1\}\ar@/^/[d]^{\xi}\ar@/^/[ll]^{\eta}\\
\{2\}\ar@/_/[rr]_{\eta}\ar@/_/[u]_{\xi}&& \{1,2\}\ar@/_/[ll]_{\eta}\ar@/^/[u]^{\xi}\\
}
\end{displaymath}
The relations are: $\eta^2=\xi^2=0 \quad\text{ and }\quad \eta\xi=\xi\eta$. 
The path category of $Q$ with the above relations will be denoted $C_E$. 

The category $C_E$ is canonically isomorphic to the tensor product $\displaystyle \bigotimes_{i\in E}C_{\{i\}}$,
where in the case of infinite $E$ the tensor product is understood as the direct limit of the directed system
formed by all tensor products with respect to finite subsets of $E$ (see e.g. \cite{Bl}).
Note that the algebra $C_{\{1\}}$ is the path category of the following quiver with relations:
\begin{equation}\label{eq22}
\xymatrix{\varnothing\ar@/^/[rr]^{\alpha}&&
\{1\}\ar@/^/[ll]^{\alpha}},
\quad\quad\alpha^2=0. 
\end{equation}
It is easy to see that in the case of finite $E$ all projective $C_E$-modules are injective and have the
same Loewy length.

\subsection{Koszulity}\label{s3.2}

Recall that a $\mathbb{Z}$-graded associative $\Bbbk$-algebra $\displaystyle
C=\bigoplus_{i\in\mathbb{Z}}C_i$ is called
{\em Koszul} if $C_0$ is semisimple, $C_i=0$ for $i<0$, and the $i$-th component of the minimal 
graded projective resolution of $C_0$ is generated in degree $i$ (such resolution is called {\em linear}). 
Similarly one defines Koszulity for $\Bbbk$-linear categories.

\begin{proposition}\label{prop21}
Let $E$ be as in the previous subsection. Then the category $C_E$ is Koszul.
\end{proposition}

\begin{proof}
For $t\in E$ the algebra $C_{\{t\}}$ is given by \eqref{eq22}. It is quadratic and monomial, hence Koszul (it 
is straightforward to write down linear projective resolutions of simple $C_{\{t\}}$-modules). 
The claim now follows from the standard observation that any tensor product of Koszul algebras is Koszul. 
To see the latter, fix linear resolutions for each simple module over every tensor factor of the product. Tensoring 
these resolutions together (one resolution per tensor factor) and taking the total complex in the usual way we 
obtain linear projective resolutions for simple $C_E$-modules. The case with finitely many factors can be
found for example in \cite{BF}. The infinite case follows by taking the direct limit with respect to the
directed system given by finite subsets of $E$.
\end{proof}

\subsection{Description of the blocks}\label{s3.3}

Finally, we would like to show that blocks of $\mathfrak{W}$ are described by $C_E$ for appropriate $E$.
Recall that for  $\mathbf{p}\in\Bbbk^{\mathtt{I}}$ we denote by 
$\mathtt{J}_{\mathbf{p}}$ the set of all $i\in \mathtt{I}$
such that $p_i\in\mathbb{Z}$. If $\mathtt{J}_{\mathbf{p}}=
\varnothing$, the projective module $P(\mathbf{p})$
is simple which means that $\mathfrak{W}_{\mathbf{p}}$
is semisimple and hence isomorphic to $\Bbbk\text{-}\mathrm{mod}$.

\begin{theorem}\label{thm27}
If $\mathtt{J}_{\mathbf{p}}\neq\varnothing$, then the category $\mathfrak{W}_{\mathbf{p}}$ is equivalent to the
category of modules over $C_{\mathtt{J}_{\mathbf{p}}}$.
\end{theorem}

\begin{proof}
Fix some representative in each isomorphism class of indecomposable projectives in $\mathfrak{W}_{\mathbf{p}}$
and let $\mathcal{X}$ be the full subcategory of $\mathfrak{W}_{\mathbf{p}}$ which these fixed representatives
generate. To prove our theorem we apply the classical Morita theory for rings with local units, see 
\cite[Theorem~4.2]{Ab}. The only nontrivial thing we have to check is that $\mathcal{X}$ is isomorphic to
$C_{\mathtt{J}_{\mathbf{p}}}$. 

For $\mathbf{m}\in\mathbf{p}+\mathbb{Z}_f^{\mathtt{I}}$ let $U(\mathbf{m})$ be the set of all $i\in \mathtt{J}_{\mathbf{p}}$ satisfying
either $p_i\in\{-1,-2,\dots\}$ while $m_i\in\{0,1,2,\dots\}$ or $p_i\in\{0,1,2,\dots\}$ while $m_i\in\{-1,-2,\dots\}$.
Then the map $P(\mathbf{m})\mapsto U(\mathbf{m})$ induces a bijection from objects of $\mathcal{X}$ to objects of
$C_{\mathtt{J}_{\mathbf{p}}}$. We realize the inverse of this map in the following way: for a finite 
$U\subset \mathtt{J}_{\mathbf{p}}$ define $\mathbf{p}^{(U)}$ as follows:
\begin{displaymath}
l^{(U)}_i:=
\begin{cases}
p_i, & p_i\not \in \mathbb{Z}\text{ or }i\not\in U;\\
0, & i\in U\text{ and }p_i<0;\\
-1, & i\in U\text{ and }p_i\geq 0.
\end{cases}
\end{displaymath}
For each $U$ fix some nonzero $v_U\in P(\mathbf{p}^{(U)})_{\mathbf{p}^{(U)}}$. 

Take now any finite $U\subset \mathtt{J}_{\mathbf{p}}$ and $i\in \mathtt{I}_{\mathbf{p}}\setminus U$.
If $p_i<0$, then define the homomorphism $\alpha_{U,i}:P(\mathbf{p}^{(U)})\to P(\mathbf{p}^{(U\cup\{i\})})$
by sending $v_U$ to $Y^{-p_i}v_{U\cup\{i\}}$, and also define the homomorphism 
$\beta_{U,i}:P(\mathbf{p}^{(U\cup\{i\})})\to P(\mathbf{p}^{(U)})$ by sending 
$v_{U\cup\{i\}}$ to $X^{-p_i}v_{U}$,
If $p_i>0$, then define the homomorphism $\alpha_{U,i}:P(\mathbf{p}^{(U)})\to P(\mathbf{p}^{(U\cup\{i\})})$
by sending $v_U$ to $X^{p_i+1}v_{U\cup\{i\}}$, and also define the homomorphism 
$\beta_{U,i}:P(\mathbf{p}^{(U\cup\{i\})})\to P(\mathbf{p}^{(U)})$ by sending 
$v_{U\cup\{i\}}$ to $Y^{p_i+1}v_{U}$. It is straightforward to check that these homomorphisms
satisfy the defining relations of $C_{\mathtt{J}_{\mathbf{p}}}$. Comparing the characters of 
indecomposable projective modules in $\mathfrak{W}_{\mathbf{p}}$ and over $C_{\mathtt{J}_{\mathbf{p}}}$
we conclude that these are all defining relations. The claim follows.
\end{proof}

\vspace{0.2cm}

\noindent
V.F.: Instituto de Matem\'atica e Estat\'istica, Universidade de S\~ao
Paulo,  S\~ao Paulo SP, Brasil; e-mail: {futorny\symbol{64}ime.usp.br}
\vspace{0.1cm}

\noindent
D.G.: Department of Mathematics, University of Texas at Arlington,  
Arlington, TX 76019, USA; e-mail: {grandim\symbol{64}uta.edu}
\vspace{0.1cm}

\noindent
V.M.: Department of Mathematics, Uppsala University, 
Box 480, SE-751 06, Uppsala, Sweden; e-mail: {mazor\symbol{64}math.uu.se}
\end{document}